\documentclass{amsproc}
\usepackage{amsmath,amssymb,amscd,mathrsfs,amscd, todonotes,appendix}
\usepackage{graphics,verbatim}
\usepackage{todonotes}
\usepackage{enumitem}
\usepackage{hyperref}
\usepackage{setspace}

\newcommand{\ind}{\operatorname{ind}}
\newcommand{\la}{\lambda}
\newcommand{\al}{\alpha}

\newcommand{\ga}{\gamma}
\newcommand{\ch}{\text{ch }}
\newcommand{\Ext}{\operatorname{Ext}}
\newcommand{\Hom}{\operatorname{Hom}}

\newcommand{\Fq}{\mathbb{F}_q}
\newcommand{\gfq}{G(\mathbb{F}_q)}

\newcommand{\str}{\text{St}_r}
\newcommand{\gr}{G_{r}}

\newtheorem{theorem}{Theorem}[subsection]
\newtheorem{lemma}[theorem]{Lemma}
\newtheorem{prop}[theorem]{Proposition}

\theoremstyle{definition}

\theoremstyle{remark}

\numberwithin{equation}{section}

\begin{document}

\title{On the Induction Functor from Group Algebras to Distribution Algebras}

\author{\sc Christopher P. Bendel}
\address
{Department of Mathematics, Statistics and Computer Science\\
University of
Wisconsin-Stout \\
Menomonie\\ WI~54751, USA}
\thanks{Research of the first author was supported in part by an AMS-Simons Research Enhancement Grant for PUI Faculty}
\email{bendelc@uwstout.edu}

\author{\sc Daniel K. Nakano}
\address
{Department of Mathematics\\ University of Georgia \\
Athens\\ GA~30602, USA}
\thanks{Research of the second author was supported in part by
NSF grant DMS-2401184}
\email{nakano@math.uga.edu}

\author{\sc Cornelius Pillen}
\address{Department of Mathematics and Statistics \\ University
of South
Alabama\\
Mobile\\ AL~36688, USA}
\email{pillen@southalabama.edu}

\date{}

\begin{abstract}
Let $G$ be a reductive algebraic group scheme defined over ${\mathbb F}_{p}$ and $k$ be an algebraically closed field of characteristic $p$. There are two associated families of finite group schemes, the $r$-th Frobenius kernels, denoted by $G_r$, and the fixed points of the iterated Frobenius map, the finite groups of Lie type, denoted by $G(\mathbb{F}_q).$ 

Bendel, Nakano and Pillen initiated the investigation of the induction functor $\ind_{G(\mathbb{F}_q)}^G-$. Using filtrations and truncation, large amounts of data coming from the algebraic group and the Frobenius kernels can be transferred to the finite group. This paper looks at connections between a fundamental theorem of Chastkofsky and Jantzen and the induction functor via the cohomology and representation theory of $G$.

\end{abstract}

\maketitle

\section{Introduction}  

\subsection{} Let $G$ be a connected reductive algebraic group scheme defined over a finite (prime) field  $
{\mathbb F}_p $ of $p$ elements and set $k = \overline{{\mathbb F}}_p $. If $F=\text{Fr}:G\rightarrow G$ is the Frobenius morphism and $F^{r}$ is the $r$th iteration of this map, then 
one can consider three distinct Hopf algebras: 
\begin{itemize} 
\item[(i)] $\text{Dist(G)}$,
\item[(ii)] $\text{Dist}(G_{r})$, 
\item[(iii)] $kG({\mathbb F}_{q})$.
\end{itemize} 
Modules for the infinite-dimensional Hopf algebra $\text{Dist}(G)$, the distribution algebra of $G$, are closely related to the rational representations of $G$ (i.e., $\text{Mod}(G)$). The infinitesimal Frobenius kernel $G_{r}$ is 
the scheme theoretic kernel of $F^{r}$, and its representation theory is equivalent to modules for the finite-dimensional Hopf algebra $\text{Dist}(G_{r})$. Finally, the $F^{r}$ fixed points on $G$ yield the 
finite Chevalley group $G({\mathbb F}_{q})$ whose representations are equivalent to those for the group algebra $kG({\mathbb F}_{q})$, which is also a finite-dimensional Hopf algebra.

A given rational $G$-module $M$ may be naturally restricted to a module for the Frobenius kernel $G_{r}$ or for the finite group $ G({\mathbb F}_q) $. Historically, one has studied the module categories for these 
Hopf algebras via the the following diagram with the restriction maps and their adjoint functors. The induction functors here are exact since the quotients $G/G_{r}$ and $G/G({\mathbb F}_{q})$ are affine. There is 
no direct functorial connection between the categories of $G_{r}$-modules and $ G({\mathbb F}_{q}) $-modules. 

\begin{figure}[ht]\label{proj1} 
\setlength{\unitlength}{.5cm}
\begin{center}
\begin{picture}(13,6)
\put(5.1,5.2){$\text{Mod}(G)$}
\put(1,1.2){$\text{Mod}(G_{r})$} 
\put(9.5,1.2){$\text{Mod}(G({\Bbb F}_{q}))$} 

\put(5.2,4.9){\vector(-1,-1){2.5}}
\put(3.1,2.2){\vector(1,1){2.5}}

\put(8.1,4.7){\vector(1,-1){2.5}}
\put(10.0,2.2){\vector(-1,1){2.5}}

\put(2.5,3.8){${\text{res}}$}
\put(4.7,2.8){$\text{ind}$}

\put(9.5,3.8){$\text{res}$}
\put(7.5,2.8){$\text{ind}$}

\end{picture}
\end{center}
\end{figure}

Using the idea of lifting to $G$-structures (e.g., identifying a compatible $G$-module structure on a given $G_r$-module), Curtis, in the 1960s \cite{C60}, provided a one-to-one correspondence between the simple $G_{r}$-modules (over $
k$) and those of $G({\mathbb F}_{q})$. The correspondence is obtained by simply restricting the simple $G$-modules that hve $p^{r}$-restricted highest weights. 
These results, with work of Steinberg \cite{St63} via the twisted tensor product theorem, allow one to transfer the questions pertaining to the computation of irreducible characters for simple $G$-modules to
the calculations of characters for $G_{1}$-modules. Given the previous results in this area, a natural course of investigation should be to 
find relationships between injective modules for $G_r$ and $ G({\mathbb F}_q) $.  The ability to lift injective $G_{r}$-modules to $G$-modules is known to hold when $p\geq 2h-4$, where $h$ is the Coxeter number (cf. \cite{BNPS24}). 
The work of Chastkofsky and Jantzen shows how the lifts of these $G_{r}$-injectives decompose as injective modules upon restriction to $G({\mathbb F}_q) $. 

This lifting approach has been one of the main ideas used to relate the representation and 
cohomology theory of algebraic groups, Frobenius kernels, and finite groups of Lie type 
(in the defining characteristics) (see, e.g.,  \cite{HV73}, \cite{CPSvdK}, \cite{Jan81}, \cite{Don93}, \cite{Hum06}, and \cite{rags}).

\subsection{} In the 1990s, Bendel, Nakano and Pillen (BNP) introduced the use of infinite-dimensional modules $\ind_{G(\mathbb{F}_q)}^G M$ into the study of the representation theory of $G$, $G_{r}$ and 
$G({\mathbb F}_{q})$. At first glance, this idea seems innocuous, but because of certain filtrations that exist on the module $\ind_{G(\mathbb{F}_q)}^G k$ via the Lang map, the BNP-approach revolutionized the field by 
answering open questions and significantly improving known results. These include 
\begin{itemize} 
\item[$\bullet$] proving that self extensions for $\gfq$ vanish \cite{BNP2, BNP3},
\item[$\bullet$] locating the first non-trivial cohomology classes for $\gfq$ \cite{BNP8,BNP9}, 
\item[$\bullet$] computing low dimensional cohomology groups for simple modules over $\gfq$ \cite{UGA2,UGA1}, 
\item[$\bullet$] bounding $\text{Ext}$-groups between simple modules for finite groups of Lie type \cite{BNPPSS}, 
\item[$\bullet$] confirming the existence of mock injective modules \cite{HNS17}. 
\end{itemize} 

In 1977, Cline, Parshall, Scott and van der Kallen \cite{CPSvdK} proved a groundbreaking result about the behavior of $G$-cohomology in relation 
to the finite group cohomology for $G({\mathbb F}_{q})$. If $V$ is a finite-dimensional rational $G$-module, then for a fixed $n\geq 0$ and $s$ and $r$ sufficiently large the restriction 
map
$$\text{res}:\operatorname{H}^{n}(G,V^{(s)})\rightarrow \operatorname{H}^{n}(G({\mathbb F}_{q}),V^{(s)})$$ 
is an isomorphism. Note that the Frobenius map is an automorphism of the finite group and induces an isomorphism on the finite group cohomology:  
$$\operatorname{H}^{n}(G({\mathbb F}_{q}),V^{(s)})\cong \operatorname{H}^{n}(G({\mathbb F}_{q}),V).$$
One important consequence of this theorem is, as $r$ increases, the cohomology groups $\text{H}^{n}(G({\mathbb F}_{q}),V)$ obtain a stable or generic value 
$\text{H}^{n}_{\text{gen}}(G,V)$ (also known as the generic cohomology). 

For over 35 years, this paper stood as the defining paper on the subject of rational and generic cohomology. The 2014 work of Bendel, Nakano and Pillen 
\cite{BNP14} utilized the functorial approach via the induction functor to prove 
the existence of generic $G$-cohomology and its stability with rational $G$-cohomology groups. New results on the vanishing of $G$ and $B$-cohomology groups were presented. 
Furthermore, vanishing ranges for the associated finite group cohomology of $G({\mathbb F}_{q})$ were  established which generalize earlier work of Hiller, in addition to (uniform) stability ranges for generic cohomology which significantly improve on seminal work of Cline, Parshall, Scott and van der Kallen. The results in \cite{CPSvdK} have been used to make computations for $G({\mathbb F}_{q})$-cohomology. However, the BNP-machinery lends itself better for this purpose because often no twisting by the Frobenius is required. 

\subsection{} In 1981, Chastkofsky and Jantzen, independently, found a formula that describes the decomposition of the Brauer character of a principal indecomposable $G_r$-module into principal indecomposable modules for $G(\mathbb{F}_q).$ The multiplicities are given in terms of data coming from the algebraic group $G$. The theorem has found many applications for the finite group $G(\mathbb{F}_q)$, including character data, cohomology, Cartan invariants and the decomposition of Deligne-Lusztig characters.

The main goal of this paper is to apply the induction functor $\text{ind}_{G({\mathbb F}_q)}^{G} -$ to provide a new proof of the relationship between projective indecomposable modules of Chastkofsky \cite{Cha81} and Jantzen \cite{Jan81} (see Theorem \ref{T:ChJan}).  The key ideas for the proof are developed in Section~\ref{S:Prior}.  We further observe that this decomposition result parallels a decomposition result for extensions over $\gfq$ (see Section~\ref{S:Ext-theorems}). Lastly,  we  provide a description of the functor applied to certain $G({\mathbb F}_{q})$-modules.

\subsection{Acknowledgements} This paper was initiated in November 2025, when the third author presented a talk in honor of Leonard Chastkofsky's retirement in the University of Georgia Algebra Seminar. 
The second author would like to thank James Zhang and the other organizers for the efforts in hosting a wonderful 5-day conference 
in Seattle in December 2025.

\section{Preliminaries}

\subsection{Notation}\label{S:notation} Throughout this paper, $k$ is an algebraically closed field of characteristic $p>0$.  We will follow for the most part 
the standard conventions in \cite{rags}. Let
\vskip .25cm 
\begin{itemize}
\item[(1)] $G$ be a connected semisimple algebraic group scheme defined over ${\mathbb F}_{p}$.
\item[(2)] $T$ be a fixed split maximal torus in $G$. 
\item[(3)] $\Phi$ be the root system associated to $(G,T)$. 
\item[(4)] $\Phi^{\pm}$ be the set of positive (resp. negative) roots. 
\item[(5)] $\Delta=\{\alpha_1,\dots,\alpha_{l}\}$ be the set of simple roots determined by $\Phi^+$. 
\item[(6)] $B$ be the Borel subgroup given by the set of negative roots, $U$ be the unipotent radical of $B$. 
\item[(7)] $W$ be the Weyl group associated with $\Phi$. 
\item[(8)]  $w_0$ denote the longest word of $W$. 
\item[(9)] $\rho$ be the half-sum of positive roots (which is also the sum of the fundamental weights). 
\item[(10)] $\alpha_0$ be the highest short root, with associated coroot $\alpha_0^{\vee}$.
\item[(11)] $h$ denote the Coxeter number for the root system associated to $G$, i.e., $h = \langle\rho,\alpha_0^{\vee}\rangle + 1$.
\item[(12)] $X:=X(T)$ be the integral weight lattice spanned by the fundamental weights $\{\omega_1,\dots,\omega_l\}$. 
\item[(13)] $X_{+}=X(T)_{+}$ denote the dominant weights for $G$. 
\item[(14)] $\leq$ be the order relation defined on $X$ via $\mu \leq \la$ if and only if  $\la - \mu = \sum_{\al \in \Delta}n_{\al}\al$ for $n_{\al} \in {\mathbb Z}_{\geq 0}$.
\item[(15)]  $X_{r}$ be the $p^{r}$-restricted weights. 
\end{itemize} 
\vskip .25cm 
\noindent 
Given a dominant weight $\nu$, define its dual weight by $\nu^* := -w_0\nu$. Note that duality is a bijection on $X_{+}$.  As such, when taking a sum over {\em all} dominant weights, weights may be exchanged for their duals; a fact that will be used at multiple points in the arguments below. The following modules naturally arise in the representation theory of $G$. For $\lambda\in X_{+}$, there are four fundamental families of finite-dimensional highest weight rational $G$-modules: 
\vskip .15cm 
$\bullet$ $L(\lambda)$ (simple), 
\vskip .15cm
$\bullet$ $\nabla(\lambda)=\text{ind}_{B}^{G}\lambda$ (costandard/induced), 
\vskip .15cm 
$\bullet$
$\Delta(\lambda)$ (standard/Weyl), 
\vskip .15cm 
$\bullet$ 
$T(\lambda)$ (indecomposable tilting).
\vskip .15cm\noindent
The Weyl module can be realized as $\Delta(\la) := \nabla(\la^*)^*$ (the linear dual of  $\nabla(\la^*)$).  There is also a family of infinite-dimensional rational $G$-modules: 
\vskip .15cm 
$\bullet$ $I(\lambda)$ injective hull of $L(\lambda)$.  
\vskip .25cm 
By $F^r$ we denote the $r$th iteration of the standard Frobenius morphism on $G$. The finite field with $q=p^r$ elements is denoted by $\Fq.$ Attached to $G$ are two families  of group schemes, the Frobenius kernels, denoted by $G_r$, and the fixed points of $F^r$, the finite groups of Lie type, denoted by $\gfq$. The simple $G_{r}$- and $\gfq$-modules are in bijective correspondence with $X_{r}$. They are realized as restrictions of simple $G$-modules with highest weights in 
$X_{r}$. For $G_r$ and $\gfq$, one has the following families of modules: 
\vskip .15cm 
$\bullet$ $L(\lambda)$ simple module for $G_{r}$ and $\gfq$,  $\lambda\in X_{r}$, 
\vskip .15cm
$\bullet$  $\widehat{Q}_r(\lambda)$ injective hull of $L(\lambda)$ as a $G_{r}T$-module,
\vskip .15cm 
$\bullet$ $U_{r}(\lambda)$ injective hull of $L(\lambda)$ as a $\gfq$-module.
\vskip .25cm 
\noindent 
Let  $\text{St}_r = L((p^r-1)\rho)$ be the $r$th Steinberg module which is a simple $G$-module and 
also projective/injective (and still simple) when restricted to $G_{r}$ and $\gfq$. 

Let $\mathbb{Z}[X]$ denote the group ring on the weight lattice $X$ and by $\mathbb{Z}[X]^W$ the subring of all $W$-invariant characters. Note that both the set  
\begin{equation}
\{\chi(\la) = \ch \nabla(\la)| \ \la \in X_{+}\}
\end{equation} and 
\begin{equation} 
\{\chi_p(\la) = \ch L(\la)| \ \la \in X_{+}\}
\end{equation} 
form  bases of $\mathbb{Z}[X]^W$. 
For any $\chi \in \mathbb{Z}[X]^W$, denote by $[\chi: \chi(\la)]_G$ the coefficient corresponding to $\chi(\la)$ when $\chi$ is expressed in the $\{\chi(\mu) |\ \mu \in X_{+}\}$ basis. Similarly, one can define $[\chi: \chi_p(\la)]_G$.

Given a $G$-module $V$ and a $p^r$-restricted weight $\la$, let $[V: L(\la)]_{\gfq}$ be the composition factor multiplicity of $L(\la)$ in $V$ when restricted to $\gfq.$   For any $\chi \in \mathbb{Z}[X]^W$ and $\la \in X_r$, set 
\begin{equation}
[\chi: L(\la)]_{\gfq} = \sum_{\mu \in X_{+}} [\chi: \chi_p(\mu)]_G \cdot [L(\mu):L(\la)]_{\gfq}.
\end{equation} 
For $\chi = \chi(\nu)$ for $\nu \in X_{+}$, we simply have 
\begin{equation} 
[\chi(\nu) : L(\la)]_{\gfq} = [\nabla(\nu) : L(\la)]_{\gfq} = [\Delta(\nu) : L(\la)]_{\gfq}.
\end{equation}


\subsection{The Induction Functor}\label{SS:BNPintro}
Given a $\gfq$-module $N$, set 
$$\mathcal{G}(N):= \ind_{\gfq}^G N.$$  
This $G$-module is infinite dimensional and encodes important information 
that relates extensions over $\gfq$ with extensions over $G$.  Indeed, from the fact that 
$G/G({\mathbb F}_q)$ is affine, one has the following key observation.

\begin{prop} \cite{BNP8}\label{P:Ext1} Let $M$ be rational $G$-module and $N$ a $\gfq$-module. Then, for all $n\geq 0$,
$$
\Ext^n_{\gfq}(M,N) \cong \Ext_G^n(M,\mathcal{G}(N)).
$$ 
\end{prop}

In addition, by Generalized Frobenius Reciprocity, one obtains

\begin{prop} \cite[Proposition 2.3.1]{BNP8}\label{P:Ext2}  Let $M, N$ be rational $G$-modules. Then, for all $n\geq 0$,
$$
\Ext^n_{\gfq}(M,N) \cong \Ext_G^n(M,N\otimes\mathcal{G}(k)).
$$ 
\end{prop}

Finally, the $G$-module $\mathcal{G}(k)$ has a special property in that it admits a filtration with sections that arise naturally from the algebra $k[G/\gfq]$. 

\begin{prop} \cite[Proposition 2.4.1] {BNP8}\label{P:Filtration} The module $\mathcal{G}(k)$ 
has a filtration with factors of the form $\nabla(\la)\otimes \nabla(\la^*)^{(r)}$ 
with multiplicity one for each $\la \in X_{+}$.
\end{prop}

\section{Prior results of Chastkofsky and Jantzen}\label{S:Prior}

\subsection{Chastkofsky's Lemma} The propositions in the following two sections are special cases of Lemma 8 in \cite{Cha81}. See also \cite[Satz 1.5]{Jan81}.
Using the filtration of Proposition \ref{P:Filtration}, one can first show the following:

\begin{prop}\label{P:chiSt1}   
Let $\chi \in \mathbb{Z}[X]^W.$ Then 
$$[\chi: \operatorname{St}_{r}]_{\gfq} = \sum_{\nu \in X_{+}}[\chi \cdot \chi(\nu): \chi((p^r-1)\rho) \cdot \chi(\nu)^{(r)}]_G.$$
\end{prop}
\begin{proof}  The statement makes sense because a theorem by Andersen and Haboush says that 
$\str \otimes \nabla(\nu)^{(r)} \cong \nabla((p^r-1)\rho + p^r \nu)$ (cf. \cite[Proposition II.3.19]{rags}).   In particular, $\chi((p^r-1)\rho)\cdot \chi(\nu)^{(r)} = \chi((p^r-1)\rho + p^r\nu)$.     
Also note that, for a fixed character $\chi$, only finitely many weights $\nu$ will contribute to the summation on the right.

It suffices to prove the statement for elements of the basis $\{\chi(\la) |\ \la \in X_+\}.$
Since $\str$ is injective and projective as a $\gfq$-module, one obtains
\begin{eqnarray*}
[\chi(\la):\str]_{\gfq} &=& [\Delta(\la) : \str]_{\gfq}\\
&=& \dim \Hom_{\gfq}({\Delta(\la),\str})\\
&=&\dim  \Hom_G(\Delta(\la), \mathcal{G}(\str)  )\\
&=&\dim  \Hom_G(\Delta(\la), \str \otimes \mathcal{G}(k) ).
\end{eqnarray*}
Now $\mathcal{G}(\str) \cong \str \otimes \mathcal{G}(k)$ is an injective $G$-module with a good filtration. Moreover, by Proposition \ref{P:Filtration}, $\mathcal{G}(k)$ has a filtration with factors $\nabla(\nu) \otimes \nabla(\nu^*)^{(r)}.$ Note that 
$$\str \otimes \nabla(\nu) \otimes \nabla(\nu^*)^{(r)} \cong \nabla(\nu) \otimes \nabla((p^r-1)\rho + p^r{\nu}^*)$$
also has a good filtration. 

Recall that given a $G$-module $M$ with a good filtration, then $\dim \Hom_G(\Delta(\la),M)$ counts the multiplicity of $\nabla(\la)$ in such a filtration. One concludes
\begin{eqnarray*}
[\chi(\la):\str]_{\gfq} 
&=&\dim  \Hom_G(\Delta(\la), \str \otimes \mathcal{G}(k) )\\
&=& \sum_{\nu \in X_{+}} \dim \Hom_G(\Delta(\la),  \nabla(\nu) \otimes \nabla((p^r-1)\rho + p^r\nu^*))\\
&=& \sum_{\nu \in X_{+}} \dim \Hom_G(\Delta(\la)\otimes \Delta(\nu^*), \nabla((p^r-1)\rho + p^r\nu^*))\\
&=& \sum_{\nu \in X_{+}} \dim \Hom_G(\Delta(\la) \otimes \Delta(\nu),  \nabla((p^r-1)\rho + p^r\nu))\\
&=& \sum_{\nu \in X_{+}} \dim \Hom_G(\Delta((p^r-1)\rho + p^r\nu),\nabla(\la) \otimes \nabla(\nu)).
\end{eqnarray*}
Observe that  $\dim \Hom_G(\Delta(\la) \otimes \Delta(\nu),  \nabla((p^r-1)\rho + p^r\nu)$ counts the multiplicity of $\Delta ((p^r-1)\rho + p^r\nu)$ in a Weyl filtration of $\Delta(\la) \otimes \Delta(\nu)$, which  equals the multiplicity of $\nabla((p^r-1)\rho + p^r\nu)$ in a $\nabla$-filtration of  $\nabla(\la) \otimes \nabla(\nu).$ This justifies the last equality. The latter is equal to
 $[\chi(\la) \cdot \chi(\nu):\chi(p^r-1)\rho + p^r \nu)]_G.$ The claim now follows.

\end{proof}

\subsection{} Next we reformulate the formula in Proposition~\ref{P:chiSt1} by performing a ``change of basis".

\begin{prop}\label{P:chiSt2}
Let $\chi \in \mathbb{Z}[X]^W.$ Then 
$$[\chi: \operatorname{St}_{r}]_{\gfq} = \sum_{\nu \in X_{+}}[\chi \cdot \chi_p(\nu): \chi((p^r-1)\rho) \cdot \chi_p(\nu)^{(r)}]_G.$$
\end{prop}
\begin{proof} First note the statement makes sense because  
$\str \otimes L(\nu)^{(r)} \cong L((p^r-1)\rho + p^r \nu).$ 

We introduce integers $a_{\nu\ga}=[\chi_p(\nu):\chi(\ga)]_G$ and $b_{\ga\nu}=[\chi(\ga):\chi_p(\nu)]_G$. Observe that $b_{\ga\nu}$ equals the multiplicity of  $\nabla(\ga)$ in a good filtration of the injective hull $I(\nu)$ of the simple $G$-module $L(\nu)$ 
via reciprocity (cf. \cite[II 4.18 Proposition]{rags}).  Consider the special case $\chi = \chi(\la)$ and see that

\begin{align*}
\sum_{\nu \in X(T)_{+}}[\chi (\la) \cdot \chi_p(\nu) &: \chi((p^r-1)\rho) \cdot \chi_p(\nu)^{(r)}]_G \\
&= \sum_{\nu \in X(T)_{+}} \sum_{\ga\in X(T)_{+}} a_{\nu\ga} [\chi(\la) \cdot \chi(\ga): \chi((p^r-1)\rho) \cdot \chi_p(\nu)^{(r)}]_G\\
&= \sum_{\nu \in X(T)_{+}} \sum_{\ga\in X(T)_{+}} a_{\nu\ga} \dim \Hom_G(\Delta(\la) \otimes  \Delta (\ga), I((p^r-1)\rho+ p^r\nu))\\
&= \sum_{\nu \in X(T)_{+}} \sum_{\ga\in X(T)_{+}} a_{\nu\ga} \dim \Hom_G(\Delta(\la) \otimes  \Delta (\ga), \str \otimes I(\nu)^{(r)})\\
&= \sum_{\nu \in X(T)_{+}} \sum_{\ga\in X(T)_{+}} b_{\ga\nu}a_{\nu\ga} \dim \Hom_G(\Delta(\la) \otimes  \Delta (\ga), \str \otimes \nabla(\ga)^{(r)})\\
&= \sum_{\ga \in X(T)_{+}} \dim \Hom_G(\Delta(\la) \otimes  \Delta (\ga), \str \otimes \nabla(\ga)^{(r)}),\\
\end{align*}
where the last equality follows from the facts that (i) $a_{\nu\ga}=0$ unless $\ga \leq \nu$, (ii) $b_{\ga\nu} \neq 0$ implies that $\nu \leq \ga$, and (iii) $a_{\nu \nu} = b_{\nu\nu}=1$.  Reidentifying this last expression in terms of multiplicities and applying Proposition \ref{P:chiSt1} gives the claim:

\begin{align*}
\sum_{\nu \in X(T)_{+}}[\chi (\la) \cdot \chi_p(\nu) &: \chi((p^r-1)\rho) \cdot \chi_p(\nu)^{(r)}]_G \\
& = \sum_{\ga \in X(T)_{+}} \dim \Hom_G(\Delta(\la) \otimes  \Delta (\ga), \str \otimes \nabla(\ga)^{(r)})\\
&= \sum_{\ga \in X(T)_{+}}[\chi (\la)\cdot \chi(\ga): \chi((p^r-1)\rho) \cdot \chi(\ga)^{(r)}]_G \\
&= [\chi(\la): \str]_{\gfq}.
\end{align*}

\end{proof}

\subsection{Jantzen's Lemma}
Let $\la \in X_r.$ The $G_r$-module $\widehat{Q}_r(\la)$ appears as a $\gr$-summand in the $G$-module tensor product $\str \otimes L((p^r-1)\rho+w_0\la)$. In general, this summand does not have a $G$-module structure, but it has the character of a $G$-module. 
Moreover, there exists a character $q_r(\la) \in \mathbb{Z}[X]^W$ with 
$$\ch \widehat{Q}_r(\la) = \ch \str \cdot q_r(\la).$$ 
The following result shows how the character $q_r(\la)$ plays a roll in expressing $\chi$ in terms of certain irreducible characters. 

\begin{lemma} \cite[2.6 Lemma]{Jan81}\label{L:Jantzen} For $\la \in X_r,$ $\nu \in X_+,$ and $\chi \in \mathbb{Z}[X]^W$ one has
$$[\chi : \chi_p(p^r\nu + \la)]_G= [\chi \cdot q_r(\la^*) : \chi_p((p^r-1)\rho + p^r\nu)]_G.$$
\end{lemma}

\section{Connections with the theorem by Chastkofsky and Jantzen}

\subsection{} 
Following \cite[2.7]{Jan81}, we assign to every $\gfq$-module $V$ a Brauer character $\Psi (V).$ These are complex class functions on $\gfq$ that vanish on $p$-singular elements. One extends these assignments to elements  $\chi \in \mathbb{Z}[X]^W$ via 
$$\Psi(\chi) = \sum_{\la \in X_r} [\chi : \chi_p(\la)]_{\gfq} \Psi(L(\la)).$$
Since the Steinberg module $\str$ is  injective as a $\gfq$-module, given a $G$-module $V$, the tensor product $\str \otimes V$ is also injective as a $\gfq$-module. The Brauer character of $\str \otimes V$ can therefore be expressed as a finite linear combination of $\Psi (U_r(\mu))$s. More generally, 
the set $\{\Psi(U_r(\mu))|\; \mu \in X_r\}$ forms a basis for the image under $\Psi$ of the ideal generated by the Steinberg character $\ch \str$ in $\mathbb{Z}[X]^W.$
For any $\chi \in \mathbb{Z}[X]^W$, one can define integers $
[\ch \str \cdot \chi:U_r(\mu)]$ via
\begin{equation}\label{brauer} \Psi(\ch \str \cdot \chi) = \sum_{\mu \in X_r} [\ch \str \cdot \chi:U_r(\mu)] \Psi(U_r(\mu)).
\end{equation}
In particular, for all $\la, \mu \in X_r$, we can define $[\widehat{Q}_r(\la) :U_r(\mu)]$ from
\begin{eqnarray*}
\Psi (\ch \widehat{Q}_r(\la)) =\Psi(\ch \str \cdot q_r(\la))
&=& \sum_{\mu \in X_r} [\ch \str \cdot q_r(\la) :U_r(\mu)] \Psi(U_r(\mu))\\
&=& \sum_{\mu \in X_r} [\widehat{Q}_r(\la) :U_r(\mu)] \Psi(U_r(\mu)).
\end{eqnarray*}
If we denote by $\langle \;,\; \rangle$ the standard inner product on class functions, then, for $\la, \mu \in X_r$,  
$\langle \Psi(L(\la)), \Psi(U_r(\mu) )\rangle = \delta_{\la \mu},$ and for any $\chi \in \mathbb{Z}[X]^W$ one obtains 
$$\langle \Psi(\chi), \Psi(U_r(\mu) )\rangle = [\chi: \chi_p(\mu)]_{\gfq}.$$
Applying the inner product to   \eqref{brauer} yields $$[\ch\str \cdot \chi :U_r(\mu)]=[\ch \str \cdot \chi: \chi_p(\mu)]_{\gfq}.$$
In addition, for any $G$-module $V$ one obtains 
\begin{eqnarray*}
[\ch \str \cdot \ch V:U_r(\mu)]&=&\dim  \Hom_{\gfq}(L(\mu), \str \otimes V)\\ &=&\dim \Hom_{\gfq}(L(\mu)\otimes V^*, \str )=[\chi_p(\mu)\cdot \ch V^*: \str]_{\gfq}.
\end{eqnarray*}
This can be generalized to any  $\chi \in \mathbb{Z}[X]^W$. So we may identify
$\displaystyle{[\ch\str \cdot \chi :U_r(\mu)]}$ with 
$\displaystyle{[\chi_p(\mu)\cdot\chi^*: \str]_{\gfq},}$
where $\chi^*$ denotes the formal dual of $\chi$. 


In particular, using the fact that $q_r(\la)^* = q_r(\la^*)$, one obtains for all $\la, \mu \in X_r$
\begin{equation}\label{E:QtoU}
[\widehat{Q}_r(\la) : U_r(\mu)] = [\ch \str \cdot q_r(\la) :U_r(\mu)]=[\chi_p(\mu)\cdot q_r(\la^*): \str]_{\gfq}.
\end{equation}

Let $\mu \in X_{+}$ be a not necessarily a $p^r$-restricted weight. Using Proposition \ref{P:chiSt2} (with $\chi = \chi_p(\mu)\cdot q_r(\la^*)$) and Lemma \ref{L:Jantzen} (with $\chi = \chi_p(\mu)\cdot\chi_p(\nu)$), one can now argue as follows:
\begin{eqnarray*}
[\chi_p(\mu)\cdot q_r(\la^*): \str]_{\gfq}
&=&
 \sum_{\nu \in X(T)_{+}}[\chi_p(\mu)\cdot q_r(\la^*) \cdot \chi_p(\nu): \chi((p^r-1)\rho) \cdot \chi_p(\nu)^{(r)}]_G\\
 &=&
 \sum_{\nu \in X(T)_{+}}[\chi_p(\mu) \cdot \chi_p(\nu): \chi_p(\la) \cdot \chi_p(\nu)^{(r)}]_G.
\end{eqnarray*}
Assuming that $\mu \in X_r$, from this and (\ref{E:QtoU}), one recovers the theorem of Chastkofsky and Jantzen.
 
 \begin{theorem} \cite[Corollary 2]{Cha81}, \cite[2.10 Corollary 2]{Jan81} \label{T:ChJan} For all $\la , \mu \in X_r$ 
 $$[\widehat{Q}_r(\la) :U_r(\mu)] = \sum_{\nu \in X_+(T)} [L(\mu) \otimes L(\nu) : L(\la + p^r\nu)]_G.$$
 \end{theorem}
 
\subsection{} \label{S:Ext-theorems} In this section, we demonstrate how the decomposition of tensor products as in Theorem~\ref{T:ChJan} arises when comparing $\text{Ext}^{1}$-groups between simple modules for 
$\gfq$ to those for extensions over $G$. First, we state a special case of Proposition~\ref{P:Ext2}. 

\begin{prop}  Let $\lambda,\mu\in X_{r}$. Then
$$\operatorname{Ext}^{1}_{G({\Bbb F}_{q})}(L(\lambda),L(\mu))\cong
\operatorname{Ext}^{1}_{G}(L(\lambda),L(\mu)
\otimes {\mathcal G}(k)).$$
\end{prop}
Recall from Propostion \ref{P:Filtration} that ${\mathcal G}(k)$ has a filtration with factors $\nabla(\nu) \otimes \nabla(\nu^*)^{(r)}.$ However, it turns out that, for $p \geq 3(h-1),$ only  those factors with weights in $\Gamma_{h}=\{\nu\in X_{+}~|~\langle \nu, \alpha_{0}^{\vee}\rangle < h\}$ contribute to the calculation. One may use a truncated version of ${\mathcal G}(k)$.
Now from \cite[Theorem 7.4]{BNP1}, if $p \geq 3(h-1)$,
this truncated module  is semisimple and isomorphic to 
$\displaystyle{
 \bigoplus_{\nu\in \Gamma_{h}} L(\nu)\otimes L(\nu^*)^{(r)}}
.$
The following theorem states the connection between $\text{Ext}^{1}$-groups that is reminiscent of the theorem of Chastkofsky and Jantzen. Note that this uses the general $\text{Ext}^{1}$-result and the structure of the appropriately truncated
${\mathcal G}(k)$ as described above. An interesting open problem would be to find other results of this nature. 

\begin{theorem} \cite[Theorem 2.5]{BNP3} For $p \geq 3(h-1)$ and $\lambda, \mu\in X_{r}$,
  $$\Ext_{\gfq}^{1}(L(\la),L(\mu)) \cong
\bigoplus_{\nu\in \Gamma_{h}} \Ext_{G}^{1}(L(\la+p^{r}\nu),L(\mu)\otimes L(\nu)),$$
where $\Gamma_{h}=\{\nu\in X(T)_{+} ~|~\langle \nu, \alpha_{0}^{\vee}\rangle < h\}.$
\end{theorem}

\subsection{} For any $\mu \in X_{r}$, the module $U_{r}(\mu)$ is an injective module for $\gfq$. Therefore, 
${\mathcal G}(U_{r}(\mu))$ is an injective $G$-module. The following result provides the number of injective indecomposable $G$-summands in ${\mathcal G}(U_{r}(\mu))$ via 
tensor product decompositions in $\gfq$. 

\begin{prop} \label{P:GonU} Let $\mu\in X_{r}$ and $\sigma\in X_{+}$. Express $\sigma=\sigma_{0}+p^{r}\sigma_{1}$ where $\sigma_{0}\in X_{r}$ and 
$\sigma_{1}\in X_{+}$. Then 
$$[{\mathcal G}(U_{r}(\mu)):I(\sigma)]=[L(\sigma_{0})\otimes L(\sigma_{1}):L(\mu)]_{\gfq}.$$
\end{prop}

\begin{proof} One can calculate $[{\mathcal G}(U_{r}(\mu)):I(\sigma)]$ by looking at the number of times $L(\sigma)$ appears in the $G$-socle of 
${\mathcal G}(U_{r}(\mu))$. This is given by 
\begin{eqnarray*} 
\dim \text{Hom}_{G}(L(\sigma), {\mathcal G}(U_{r}(\mu)))&=& \dim \text{Hom}_{\gfq}(L(\sigma), U_{r}(\mu))\\
&= & \dim \text{Hom}_{\gfq}(L(\sigma_{0})\otimes L(\sigma_{1})^{(r)}, U_{r}(\mu))\\
&= & \dim \text{Hom}_{\gfq}(L(\sigma_{0})\otimes L(\sigma_{1}), U_{r}(\mu))\\
&=& [L(\sigma_{0})\otimes L(\sigma_{1}):L(\mu)]_{\gfq}.
\end{eqnarray*}
\end{proof}

\subsection{} The $G$-structure on the image of ${\mathcal G}$ on an injective $\gfq$-module can be made more transparent by using the theorem of Chastkofsky and Jantzen.
Since all the terms on the right hand side of Theorem~\ref{T:ChJan} are nonnegative, one obtains an injective $\gfq$-module, denoted by $\bar{Q}_r(\la)$, whose Brauer character equals that of the original $\widehat{Q}_r(\la),$ via  
$$\bar{Q}_r(\la) :=\bigoplus_{\mu \in X_r}
U_r(\mu)^{ \left(\sum_{\nu \in X_+} [L(\mu) \otimes L(\nu) : L(\la + p^r\nu)]_G\right)}.$$
For any dominant weight  $\mu$ one obtains from above
\begin{eqnarray*}
\dim \Hom_{\gfq}(L(\mu), \bar{Q}(\la))&=&[\chi_p(\mu)\cdot q_r(\la^*): \str]_{\gfq}\\
&=&
 \sum_{\nu \in X_+}[\chi_p(\mu) \cdot \chi_p(\nu): \chi_p(\la) \cdot \chi_p(\nu)^{(r)}]_G.\\
 &=& \sum_{\nu \in X_+} \dim \Hom_G(L(\mu) \otimes L(\nu), I(\la +p^r\nu))\\
 &=&  \dim \Hom_G(L(\mu),\bigoplus_{\nu \in X_+} ( L(\nu)\otimes I(\la +p^r\nu^*))).
\end{eqnarray*} 
On the other hand, the module $\mathcal{G}(\bar{Q}_r(\la))$ now makes sense. It is an infinite dimensional injective $G$-module. Observe that for $\la \in X_r$ and $\mu \in X_+$
\begin{eqnarray*}
\dim\Hom_G(L( \mu),\mathcal{G}(\bar{Q}_r(\la)) &=& \dim \Hom_{\gfq}(L(\mu), \bar{Q}(\la))\\
&=&
 \dim \Hom_G(L(\mu),\bigoplus_{\nu \in X_+} ( L(\nu)\otimes I(\la +p^r\nu^*))).
\end{eqnarray*} 
Hence, $\mathcal{G}(\bar{Q}_r(\la))$ and $\bigoplus_{\nu \in X_+} ( L(\nu)\otimes I(\la +p^r\nu^*))$
are injective $G$-modules with isomorphic $G$-socles. One concludes the following from \cite[I.3.13]{rags}.

\begin{theorem}\label{P:GQbar} For $\la \in X_r$ and $\bar{Q}_r(\la)$ as defined  above, 
\begin{itemize}
\item[(a)] There is an isomorphism of $G$-modules:
$$\mathcal{G}(\bar{Q}_r(\la)) \cong \bigoplus_{\nu \in X_+} ( L(\nu)\otimes I(\la +p^r\nu^*)).$$
\item[(b)] If $\widehat{Q}_r(\la)$ has a G-module structure, one has 
$$\mathcal{G}(\widehat{Q}_r(\la)) \cong \bigoplus_{\nu \in X_+} ( \widehat{Q}_r(\la)\otimes  L(\nu)\otimes I(\nu^*)^{(r)}).$$
\item[(c)] There is an isomorphism of $G$-modules:
$$\mathcal{G}(\operatorname{St}_{r}) \cong \bigoplus_{\nu \in X_+} ( \operatorname{St}_{r} \otimes L(\nu)\otimes I(\nu^*)^{(r)}).$$
\end{itemize} 
\end{theorem} 

\begin{proof} Part (a) follows by the aforementioned discussion. Part (b) is an immediate consequence of (a) and \cite[II.11.16 Proposition]{rags}. Part (c) is a special case of (b). 
\end{proof}

\providecommand{\bysame}{\leavevmode\hboxto3em{\hrulefill}\thinspace}

\end{document}